\documentclass[sn-mathphys,Numbered]{sn-jnl}
\usepackage{graphicx}%
\usepackage{multirow}%
\usepackage{amsmath,amssymb,amsfonts}%
\usepackage{amsthm}%
\usepackage{mathrsfs}%
\usepackage[title]{appendix}%
\usepackage{xcolor}%
\usepackage{textcomp}%
\usepackage{manyfoot}%
\usepackage{booktabs}%
\usepackage{algorithm}%
\usepackage{algorithmicx}%
\usepackage{algpseudocode}%
\usepackage{listings}%
\theoremstyle{thmstyleone}%
\newtheorem{theorem}{Theorem}% 
\newtheorem{proposition}[theorem]{Proposition}% 
\theoremstyle{thmstyletwo}%
\usepackage{bm}
\usepackage{float}
\usepackage{booktabs}
\usepackage{epsfig}
\usepackage{graphicx}
\usepackage{caption}
\usepackage{color}
\usepackage{ulem}
\usepackage{relsize}
\usepackage{exscale}
\usepackage{xcolor}
\usepackage{varioref}
\usepackage{comment}
\usepackage{bbm}
\usepackage{multirow}

\usepackage{subcaption}
\usepackage{graphicx}
\usepackage{float}   

\theoremstyle{thmstylethree}%
\usepackage{actuarialangle}

\newtheorem{ex}{Example}

\newtheorem{lem}{Lemma}

\usepackage{eurosym}

\def\annu#1{_{% 
\vbox{\hrule height .2pt 
\kern 1pt 
\hbox{$\scriptstyle {#1}\kern 1pt$}% 
}\kern-.05pt 
\vrule width .2pt 
}}

\begin{document}

\title[Constant]{Several expressions of the net single premiums under the constant force of mortality}

\author*[1]{\fnm{Andrius} \sur{Grigutis}}\email{andrius.grigutis@mif.vu.lt}
\author[1]{\fnm{Laurynas} \sur{Lukoševičius}}\email{laurynas.lukosevicius@mif.stud.vu.lt}
\author[1]{\fnm{Mindaugas}\sur{Venckevičius}}\email{mindaugas.venckevicius@mif.stud.vu.lt}

\affil*[1]{\orgdiv{Institute of Mathematics}, \orgname{Vilnius University}, \orgaddress{\street{Naugarduko 24}, \city{Vilnius}, \postcode{LT-03225}, \country{Lithuania}}}

\abstract{In this article, we present several formulas that make it easier to compute the net single premiums when the mortality force over the fractional ages is assumed to be constant (C). More precisely, we compute the moments of the random variables $\nu^{T_x}$, $T_x$, $T_x\nu^{T_x}$, etc., where $T_x$ denotes the future lifetime of a person who is $x\in\{0,\,1,\,\ldots\}$ years old, and $\nu$ is the annual discount multiplier. We verify the obtained formulas on the real data from the human mortality table and the Gompertz survival law. The obtained numbers are compared with the corresponding ones when the survival function over fractional ages is interpolated using the uniform distribution of deaths (UDD) and Balducci's (B) assumptions. We also formulate and prove the statement on the comparison of the moments of the mentioned random variables under assumptions (C), (UDD), and (B).}




\maketitle

\section{Introduction}\label{sec:intr}
Let $X$ be the continuous and non-negative random variable that describes a person's life-time.  In actuarial mathematics, the tail of the distribution function
\begin{align*}
s(y):=\mathbb{P}(X> y),\,y\geqslant0
\end{align*}
is called the {\bf survival function}. If there exists a derivative $s'(y)$, then the ratio
\begin{align*}
\mu(y):=-\frac{s'(y)}{s(y)},\,y\geqslant0.
\end{align*}
is called the {\bf force of mortality}. It is well known (see, for example, \cite[p. 49]{bowers1986actuarial}) that
\begin{align*}
\frac{s(x+t)}{s(x)}=\exp\left(-\int_{x}^{x+t}\mu(y)\,dy\right),
\end{align*}
where $x\geqslant0$ is a person's age and $t\geqslant0$ determines the future life time. In practice, the values of the survival function are often known for an integer age only, i.e. $s(x)$ is given for $x\in\{0,\,1,\,\ldots\}=:\mathbb{N}_0$ only. If $x\in\mathbb{N}_0$ is fixed, then the values of the survival function for fractional ages, i.e., $s(x+t)$, $0<t<1$, are computed according to certain agreements or laws. The statistical mortality tables are typically built in a way that mortality is commonly described by the conditional annual one-year death probability $q_x$ or the corresponding survival probability $p_x=1-q_x$. Since these quantities are usually available only for integer ages, an interpolation assumption is required to define survival over fractional ages. Different assumptions yield different continuous-time survival functions and forces of mortality. There are three main fundamental laws of interpolation to connect the points $s(x),\, s(x+1),\,s(x+2),\,\ldots$:

\bigskip

{\bf Uniform distribution of deaths} (UDD)
\begin{align}\label{UDD}
s(x+k+t)=(1-t)s(x+k)+ts(x+k+1),\,x,\,k\in\mathbb{N}_0,\,t\in[0,\,1].
\end{align}

{\bf Constant force of mortality} (C),
\begin{align}\label{C}\nonumber
s(x+k+t)&=\left(s(x+k)\right)^{1-t}\left(s(x+k+1)\right)^t\\
&=s(x+k)\left(p_{x+k}\right)^t
,\,x,\,k\in\mathbb{N}_0,\,t\in[0,\,1].
\end{align}

{\bf Balducci's} assumption (B)
\begin{align}\label{Balducci}
\frac{1}{s(x+k+t)}=\frac{1-t}{s(x+k)}+\frac{t}{s(x+k+1)},\,x,\,k\in\mathbb{N}_0,\,t\in[0,\,1].
\end{align}

In Figure \ref{f1_}, we connect $s(0)=1,\,s(1)=0.8,\,s(2)=0.7,\,s(3)=0.5,\,s(4)=0.2$ according to the these three interpolations.

\begin{figure}[H]
\centering
\includegraphics[scale=0.75]{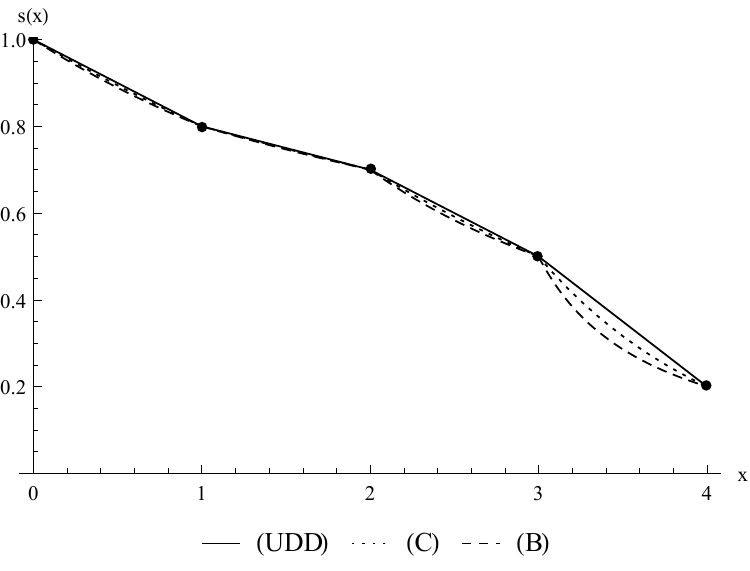}
\caption{Example of $s(x)$ interpolation according to (UDD), (B), and (C).}
\label{f1_}
\end{figure}

As seen from Figure \ref{f1_}, if $s(x+1)$ is just slightly lower than $s(x)$, the effect of interpolation is miserable. However, the influence of interpolation can become more significant when computing certain characteristics of the present value of the insured amount, particularly when the insured amount is large or the insurance period is long. See \cite{BenjaminPollard1980} as one of the earliest systematic discussions of fractional age assumptions.

Let $T_x$ be the random variable that determines the future lifetime of a person who is already $x\in\mathbb{N}_0$ years old. The future lifetime is the conditional random variable $T_x=X-x$ given that $X> x$, where $X$ is the continuous random variable that determines an entire person's lifetime. The tail distribution of $T_x$ is 
\begin{align*}
_{u}p_x:=\mathbb{P}(T_x> u)=\mathbb{P}(X> x+u|X> x)=\frac{s(x+u)}{s(x)},\,x\in\mathbb{N}_0,\,u\geqslant0,
\end{align*}
and the probability density of $T_x$ is
\begin{align}\label{cond_density}
f_x(u):=-\left( _{u}p_x\right)'_u,\,u\geqslant0,
\end{align}
if the derivative exists.

The conditional density $f_x(k+t)$ when $x\in\mathbb{N}_0$ is fixed, $k$ varies over $\{0,\,1,\,\ldots\}$, and $0<t<1$, under the interpolation assumptions (UDD), (C), and (B) respectively is:
{\bf }
\begin{align}\label{cond_dens_UDD}
&{\textbf{(UDD)}}:\quad f_x(k+t)=-(_{k+t}p_x)'_t={_k}p_x-{_{k+1}}p_x=\frac{d_{x+k}}{l_x}=:{_{k|1}}q_x,\\ \label{cond_dens_C}
&{\textbf{(C)}}:\quad f_x(k+t)=-(_{k+t}p_x)'_t
={_kp}_x\cdot(p_{x+k})^t \cdot \log (1/p_{x+k}),\\ \label{cond_dens_B}
&{\textbf{(B)}}:\quad f_x(k+t)=-(_{k+t}p_x)'_t
=\frac{_{k+1}p_x\cdot q_{x+k}}{\left(1-(1-t)\cdot q_{x+k}\right)^2},
\end{align}
where $l_x$ in \eqref{cond_dens_UDD} represents the expected number of survivors to age $x$ from the $l_0$ newborns, i.e., $l_x=l_0s(x)$, and $d_x=l_x-l_{x+1}$ denote the number of deaths between ages $x$ and $x+1$. See Figure \ref{f2_} for the illustrated densities \eqref{cond_dens_UDD}, \eqref{cond_dens_C}, \eqref{cond_dens_B} when $s(0)=1,\,s(1)=0.8,\,s(2)=0.7,\,s(3)=0.5,\,s(4)=0.2$.

\begin{figure}[H]
\centering
\includegraphics[scale=0.75]{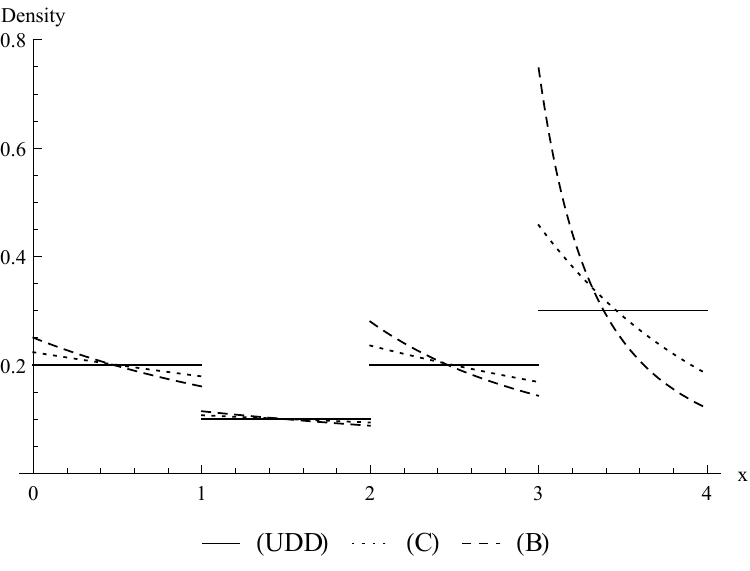}
\caption{View of the conditional density $f_x(k+t)$ under (UDD), (C), and (B) assumptions when when $s(0)=1,\,s(1)=0.8,\,s(2)=0.7,\,s(3)=0.5,\,s(4)=0.2$.}
\label{f2_}
\end{figure}

Let us shortly describe the service of mathematics in actuarial science. If \euro\, denotes the insured amount and insurance holds till the end of life, then the present value of \euro\, is given by \euro$\nu^{T_x}$, where $\nu=1/(1+i)$ denotes the discount multiplier, $i>-1$ is the annual interest rate, and $T_x$ is the future lifetime of a person who is already $x$ years old. The mathematical essence of actuarial mathematics consists of the characterization of the random variable \euro$\nu^{T_x}$ (see, for example, \cite[Ex. 20.20]{KlugmanPanjerWillmot2012}). More precisely, we may seek to compute the expected value of 
\begin{align}\label{expectation}
\mathbb{E}\nu^{T_x}=\int\limits_{0}^{\infty}\nu^uf_x(u)\,du=\int\limits_{0}^{1}\nu^uf_x(u)\,du+
\int\limits_{1}^{2}\nu^u f_x(u)\,du+\ldots
\end{align}
The computation of \eqref{expectation} or similar integrals when $\nu^{T_x}$ gets replaced by some other random function, is the easiest under the (UDD) assumption because the density $f_x$ is the step function over the intervals $(0,\,1]\,(1,\,2],\,\ldots$. This is done in various sources, see, for example, \cite{Book}. The expectations of type \eqref{expectation} under Balducci's assumption (B) are computed in \cite{Balducci}. In this paper, we derive formulas that facilitate the easier computation of the expectations of type \eqref{expectation} under the assumption of a constant force of mortality (C).

The constant force of mortality is definitely a realistic assumption in human populations. According to \eqref{C}, $_tp_{x}=(p_{x})^t$, $x\in\mathbb{N}_0$, $0\leqslant t \leqslant 1$, which shows that the corresponding one-year survival probability and the future life-time fully determine survival over fractional ages. This simplistic approximation is consistent with observed mortality patterns over short age intervals: as illustrated in Figures \ref{f1} and \ref{f2} below, the survival function under the constant force of mortality assumption is very close to that obtained from the Gompertz mortality model (\cite{gompertz1825nature}, \cite{CairnsBlakeDowd2008}) over fractional ages. In our work, the force of mortality is not the same constant throughout the entire future lifetime; it can vary over different fractional ages, see \cite[Ch. 3.3.2]{dickson2019}. However, if the force of mortality is considered as the same constant $\lambda>0$ throughout the entire future lifetime, the future lifetime is distributed exponentially with parameter $\lambda>0$, \cite{DenuitDhaeneGoovaertsKaas2005}. See also \cite[Ch. 8.10.1]{david2015}, \cite[p. 68]{bowers1986actuarial} for some actuarial characteristics when the force of mortality is considered as the same constant throughout the entire future lifetime. 

\section{Main results}\label{sec:results}

    This section consists of the main results of the paper. In Propositions \ref{prop:main}--
\ref{prop:[T]_nu} we compute the $m$-th moments of the random variables $\nu^{T_x}$, $T_x$, $T_x\cdot \nu^{T_x}$, $[T_x+1]\cdot\nu^{T_x}$ respectively. We then formulate two lemmas: one comparing the expectations of the mentioned random variables, and another on the probability that the future lifetime $T_x$ falls into the shorter-than-one-year interval. In our last statements, in Propositions \ref{prop:j_times} and \ref{prop:j_[T]_nu_x}, we divide each year in $j\in\mathbb{N}$ equal pieces and compute the $m$-th moments of the random variables $\nu^{\left([T_x j]+1\right)/j},\, [j T_x+1]\cdot\nu^{T_x}$. All of the formulated statements are proved in Section \ref{sec:proofs}. When necessary, in the remarks below the formulated statements, we examine cases where uncertainty may arise in the provided formulas. That typically consists of cases $p_{x+k}\to0$, $p_{x+k}\to1$, or $\nu^m \cdot p_{x+k}\to1$. 

We anticipate that our considered expectations cover the majority of insurance products or can be modified in desired ways to express the net actuarial values for other insurance products; see the end of this section for a more precise example.

Let us refer to \cite{Balducci} or \cite{notations} for international actuarial notations used in the further text. We start with the statement on the $m$-th moment of the random variable $\nu^{T_x}$.  

\bigskip

\begin{proposition}\label{prop:main}
Say that the survival function $s(x)$, $x\in\mathbb{N}_0$ is interpolated according to the constant force of mortality assumption \eqref{C} and let $T_x$ denote the future lifetime of a person being of $x\in\mathbb{N}_0$ years old. If $p_{x+k}>0$ for all $l\leqslant k \leqslant l+n-1$, where $l\in\mathbb{N}_0$ and $n\in\mathbb{N}$, then
\begin{align}
^m_{l|}\bar{A}_{x:\actuarialangle{n}}^1&:=\mathbb{E}\nu^{mT_x}\mathbbm{1}_{\{l\leqslant T_x < l+n\}}\nonumber\\ 
&=\sum_{k=l}^{n+l-1}\frac{\nu^{mk}\cdot {_kp}_x\cdot(1-\nu^m\cdot p_{x+k})\cdot \log p_{x+k}}{\log (\nu^m \cdot p_{x+k})},\label{for_n_years_l_delay}
\\
^m_{l|}\bar{A}_{x}&:=\lim_{n\to\infty}\, ^m_{l|}\bar{A}_{x:\actuarialangle{n}}^1,
\label{till_the_end_of_life_l}
\end{align}
where $m\in\mathbb{N}$.
\end{proposition}

\bigskip

{\sc Remark 1:} {\it In formula \eqref{till_the_end_of_life_l} $n$ runs up to such large as long as $s(x+n+l-1)>0$. The same is understood in further expected values.
} 

{\sc Remark 2:} {\it If $p_{x+k}\to 0$ in \eqref{for_n_years_l_delay}, then the corresponding summands are $\nu^{mk}\cdot{_k}p_x$. If $\nu^m \cdot p_{x+k}\to 1$ in \eqref{for_n_years_l_delay}, then the corresponding summands are $\nu^{mk}\cdot {_k}p_x\cdot \log (1/p_{x+k})$.
} 

\bigskip

Let us define the upper incomplete gamma function
\begin{align}\label{gamma}
\Gamma(a,\,x)=\int_{x}^{\infty}t^{a-1}e^{-t}\,dt,\,x\geqslant0,\,a>0.
\end{align}

\begin{proposition}\label{prop:T_x}
Say that the survival function $s(x)$, $x\in\mathbb{N}_0$ is interpolated according to the constant force of mortality assumption \eqref{C} and let $T_x$ denote the future lifetime of a person being of $x\in\mathbb{N}_0$ years old. If $0<p_{x+k}<1$ for all $l\leqslant k \leqslant l+n-1$, where $l\in\mathbb{N}_0$, $n\in\mathbb{N}$, then
\begin{align}\label{T_x_m}
\mathbb{E}\left(T_x\right)^m\mathbbm{1}_{\{l\leqslant T_x< n+l\}}
=\sum_{k=l}^{n+l-1}\frac{{_k}p_x\cdot\Gamma_{m,\,k}}{(p_{x+k})^k(\log 1/p_{x+k})^m},
\end{align}
where $m\in\mathbb{N}$,
\begin{align*}
\Gamma_{m,\,k}=\Gamma\left(m+1,\,k\log1/p_{x+k}\right)-\Gamma\left(m+1,\,(k+1)\log1/p_{x+k}\right),
\end{align*}
and $\Gamma(a,\,x)$ is the incomplete upper gamma function \eqref{gamma}.

In particular, if $m=0$, $m=1$, or $m=2$, then
\begin{align}
\mathbb{E}\mathbbm{1}_{\{l\leqslant T_x< n+l\}}&=\sum_{k=l}^{l+n-1} {_k}p_x\cdot q_{x+k}=\frac{1}{l_x}\sum_{k=l}^{l+n-1}d_{x+k}=\frac{l_{x+l}-l_{x+l+n}}{l_x},\label{0}\\
{{_{l|}}\mathop{e}\limits^\circ}_{x:\actuarialangle{n}}:=\mathbb{E}T_x\mathbbm{1}_{\{l\leqslant T_x< n+l\}}&=\sum_{k=l}^{n+l-1}{_k}p_x\cdot\frac{q_{x+k}+(kq_{x+k}-p_{x+k})\log 1/p_{x+k}}{\log 1/p_{x+k}},\label{1}
\end{align}
\begin{align}\label{2}
&\mathbb{E}T_x^2\mathbbm{1}_{\{l\leqslant T_x< n+l\}}\\ \nonumber
&=\sum_{k=l}^{n+l-1}{_k}p_x\cdot
\frac{2q_{x+k}-2(q_{x+k}\cdot k-p_{x+k})\log p_{x+k}+(k^2-(1+k)^2\cdot p_{x+k})(\log p_{x+k})^2}{(\log p_{x+k})^2}
\end{align}
\end{proposition}

{\sc Remark 3:} {\it If $p_{x+k}\to0$ in \eqref{T_x_m}, then the corresponding summands are $k^m\cdot {_k}p_x$. If $p_{x+k}\to1$, then the corresponding summands are zeros.
} 

\bigskip

    In the following Proposition, we compute the $m$-th moment of the random variable $T_x\cdot \nu^{T_x}$, which describes the present value of uniformly increasing insurance payoff. Again, the general expression is complicated, and we explicitly write down only the first two moments. 

\bigskip

\begin{proposition}\label{increasing}
Say that the survival function $s(x)$, $x\in\mathbb{N}_0$ is interpolated according to the constant mortality force assumption \eqref{C} and let $T_x$ denote the future lifetime of a person who is $x\in\mathbb{N}_0$ years old. If $0<p_{x+k}<1$ for all $l\leqslant k \leqslant l+n-1$, where $l\in\mathbb{N}_0$, $n\in\mathbb{N}$, then
\begin{align}
&_{l|}^m\left(\bar{I}\bar{A}\right)_{x:\actuarialangle{n}}^1:=\mathbb{E}\,(T_x\nu^{T_x})^m\mathbbm{1}_{\{l\leqslant T_x< n+l\}}=\sum_{k=l}^{n+l-1}\frac{{_k}p_x\cdot\log \frac{1}{p_{x+k}}\cdot\tilde{\Gamma}_{m,\,k}}{(p_{x+k})^k\cdot\left(\log \frac{1}{\nu^m\cdot p_{x+k}}\right)^{m+1}},\label{increasing_gen}
\end{align}
where $m\in\mathbb{N}$,
\begin{align*}
\tilde{\Gamma}_{m,\,k}=\Gamma\left(m+1,\,k\log\frac{1}{\nu^m\cdot p_{x+k}}\right)-\Gamma\left(m+1,\,(k+1)\log\frac{1}{\nu^m \cdot p_{x+k}}\right),
\end{align*}
and $\Gamma(a,\,x)$ is the incomplete upper gamma function \eqref{gamma}.

In particular, if $m=1$, then
\begin{align}\label{increasing_1}\nonumber
&_{l|}\left(\bar{I}\bar{A}\right)_{x:\actuarialangle{n}}^1
=\mathbb{E}\,T_x\nu^{T_x}\mathbbm{1}_{\{l\leqslant T_x< n+l\}}=\\
&\sum_{k=l}^{n+l-1}{_k}p_x\cdot\log \frac{1}{p_{x+k}}\cdot\frac{\nu^k\left(1-\nu\cdot p_{x+k}+(-k+(k+1)\cdot\nu\cdot p_{x+k})\log(\nu\cdot p_{x+k})\right)}{\left(\log(\nu p_{x+k})\right)^2}.
\end{align}

Moreover, if $m=2$, then
\begin{align}\nonumber
&^2_{l|}\left(\bar{I}\bar{A}\right)_{x:\actuarialangle{n}}^1=\mathbb{E}\left(T_x\nu^{T_x}\right)^2\mathbbm{1}_{\{l\leqslant T_x< n+l\}}\\ \nonumber
&=\sum_{k=l}^{n+l-1}{_k}p_x\cdot\log\frac{1}{p_{x+k}}\cdot\nu^{2k}\Bigg(
\frac{-2 + 2p_{x+k}\nu^2 }{(\log(\nu^2p_{x+k}))^3}+
\\
&
\frac{\log(\nu^2p_{x+k})\left((2k - 2p_{x+k}(1 + k)\nu^2 + 
(-k^2 + p_{x+k}(1 + k)^2\nu^2)\log(\nu^2p_{x+k})\right)}{(\log(\nu^2p_{x+k}))^3}
\Bigg).
\label{increasing_2}
\end{align}
\end{proposition}

{\sc Remark 4: }{\it If $p_{x+k}\to0$ in \eqref{increasing_gen}, then the corresponding summands are $\nu^{mk}\cdot k^m\cdot{_k}p_x$. If $p_{x+k}\to1$ in \eqref{increasing_gen}, then the corresponding summands are zeros. If $\nu^m\cdot p_{x+k}\to1$, then
\begin{align*}
\frac{\tilde{\Gamma}_{m,\,k}}{\left(\log \frac{1}{\nu^m\cdot p_{x+k}}\right)^{m+1}}\to
\frac{(k+1)^{m+1}-k^{m+1}}{m+1}.
\end{align*}
}

\bigskip

In the following proposition, we compute the $m$-th moment of the random variable $[T_x+1]\cdot\nu^{T_x}$, which describes the present value of the yearly increasing payoff of the insured amount.

\bigskip

\begin{proposition}\label{prop:[T]_nu}
Say that the survival function $s(x)$, $x\in\mathbb{N}_0$ is interpolated according to the constant mortality force assumption \eqref{C} and let $T_x$ denote the future lifetime of a person being of $x\in\mathbb{N}_0$ years old. If $p_{x+k}>0$ for all $l\leqslant k \leqslant l+n-1$, where $l\in\mathbb{N}_0$ and $n\in\mathbb{N}$, then
\begin{align}\label{eq:in_prop_4}\nonumber
^m_{l|}\left(I\bar{A}\right)_{x:\actuarialangle{n}}^1:&=
\mathbb{E}\,([T_x+1]\nu^{T_x})^m\mathbbm{1}_{\{l\leqslant T_x< n+l\}}\\
&=\sum_{k=l}^{l+n-1}\nu^{mk}\cdot(k+1)^m\cdot{_k}p_x\cdot\log\frac{1}{p_{x+k}}\cdot\frac{p_{x+k}\cdot\nu^m-1}{\log(\nu^m\cdot p_{x+k})}.
\end{align}
\end{proposition}

{\sc Remark 5: }{\it If $p_{x+k}\to0$ in \eqref{eq:in_prop_4}, then the corresponding summands are $\nu^{mk}\cdot(k+1)^m\cdot {_k}p_x$. If $\nu^m\cdot p_{x+k}\to1$, then these summands are $\nu^{mk}\cdot(k+1)^m\cdot {_k}p_x\cdot \log(1/p_{x+k})$.
}

\bigskip
In the next statement, Lemma \ref{lem:no_greater}, we give a precise comparison for the values of the previously considered expectations when they are computed under (UDD), (C), and (B) interpolations.
\bigskip

\begin{lem}\label{lem:no_greater}
Let $g(t)$ be the real, differentiable, and non-increasing function over the intervals $t\in[0,\,1)\cup [1,\,2),\, \ldots$ Then for the expected value of $g(T_x)$ holds
\begin{align}\label{exp_ineq_1}
\mathbb{E}_{\textup{(UDD)}}g(T_x)\leqslant\mathbb{E}_{\textup{(C)}}g(T_x)\leqslant\mathbb{E}_{\textup{(B)}}g(T_x).
\end{align}
Conversely, suppose that the function $g(t)$ is non-decreasing under the same conditions. In that case, 
\begin{align}\label{exp_ineq_2}
\mathbb{E}_{\textup{(UDD)}}g(T_x)\geqslant\mathbb{E}_{\textup{(C)}}g(T_x)\geqslant\mathbb{E}_{\textup{(B)}}g(T_x).
\end{align}
\end{lem}

\bigskip
We now divide every entire year of insurance into $j\in\mathbb{N}$ equal pieces of length $1/j$: 
\begin{align}\label{intervals}
\hspace{-0.5cm}
\begin{cases}
\text{1'th year:}\left[0,\,\frac{1}{j}\right),\,\left[\frac{1}{j},\,\frac{2}{j}\right),\,\ldots,\,
\left[1-\frac{1}{j},\,1\right),\\
\text{2'nd year:}\left[1,\,1+\frac{1}{j}\right),\,\left[1+\frac{1}{j},\,1+\frac{2}{j}\right),\,\ldots,\,
\left[2-\frac{1}{j},\,2\right),\\
\vdots\\
\text{n'th year:}\left[n-1,\,n-1+\frac{1}{j}\right),\,\left[n-1+\frac{1}{j},\,n-1+\frac{2}{j}\right),\,\ldots,\,
\left[n-\frac{1}{j},\,n\right).
\end{cases}
\end{align}

In Proposition \ref{prop:j_times}, we compute the $m$-th moment of the random variable $\nu^{([T_x j]+1)/j}$, $j\in\mathbb{N}$, where the future lifetime $T_x$ is distributed over the intervals in \eqref{intervals}. Now, the insurance deferment may consist of an entire year plus some part of the year. We denote ''$l*n_1$'', where $l\in\mathbb{N}_0$ provides years, and $n_1\in\{0,\,1,\,\ldots,\,j-1\}$ means the number of periods whose length is $1/j$. For instance, if $j=6$, then $l=2$ and $n_1=2$ describe the deferment of two years and four months. 

The next little Lemma is needed to express the probability of $T_x$ falling into the short interval under the constant force of mortality assumption (C).  

\bigskip

\begin{lem}\label{lem:j_times}
Let $x,\,k\in\mathbb{N}_0$, $j\in\mathbb{N}$, and $d=0,\,1,\,\ldots,\,j-1$. Then, under the constant force of mortality assumption \eqref{C},
\begin{align*}
{_{k+\frac{d}{j}}}p_x-{_{k+\frac{d+1}{j}}}p_x
={_k}p_x\cdot \left(p_{x+k}\right)^{\frac{d}{j}}\left(1-\left(p_{x+k}\right)^{\frac{1}{j}}\right).
\end{align*}
\end{lem}

Let us observe that Lemma \ref{lem:no_greater} is not necessarily valid for intervals shorter than one year; it does not apply to Propositions \ref{prop:j_times} and \ref{prop:j_[T]_nu_x}.

\bigskip

\begin{proposition}\label{prop:j_times}
Say that the survival function $s(x)$, $x\in\mathbb{N}_0$ is interpolated according to the constant mortality force assumption \eqref{C} and let $T_x$ denote the future lifetime of a person being of $x\in\mathbb{N}_0$ years old. If $j,\,n\in\mathbb{N}$, $m,\,l\in\mathbb{N}_0$, and $n_1\in\{0,\,1,\,\ldots,\,j-1\}$, then
\begin{align}\label{prop5_eq1}\nonumber
&^m_{l*n_1|}\left(A^{(j)}\right)_{x:\actuarialangle{n}}^1:=
\mathbb{E}\left(\nu^{\frac{[T_x\,j]+1}{j}}\right)^m\mathbbm{1}_{\{l*n_1\leqslant T_x<n+l*n_1\}}\\
&=\sum_{k=l}^{n+l-1}\nu^{mk}\cdot{_k}p_x\cdot\left(1-\left(p_{x+k}\right)^{\frac{1}{j}}\right)\sum_{d=n_1}^{j-1}\nu^{\frac{(d+1)m}{j}}\left(p_{x+k}\right)^\frac{d}{j}
\\
&+
\nu^{(n+l)m}\cdot{_{n+l}p}_x\left(1-(p_{x+n+l})^\frac{1}{j}\right)\sum_{d=0}^{n_1-1}\nu^{\frac{(d+1)m}{j}}\left(p_{x+n+l}\right)^\frac{d}{j}.
\end{align}
\end{proposition}

In the last proposition, we provide the formula to compute the $m$-th moment of the random variable $[j\cdot T_x+1]\cdot\nu^{T_x}$, $j\in\mathbb{N}$, where the future lifetime $T_x$ again is distributed over the intervals in \eqref{intervals}. Such a random variable describes the present value of the insured amount that increases $j$ times per year.

\bigskip

\begin{proposition}\label{prop:j_[T]_nu_x}
Say that the survival function $s(x)$, $x\in\mathbb{N}_0$ is interpolated according to the constant force of mortality assumption \eqref{C} and let $T_x$ denote the future lifetime of a person being of $x\in\mathbb{N}_0$ years old. If $j,\,n\in\mathbb{N}$, $m,\,l\in\mathbb{N}_0$, $n_1\in\{0,\,1,\,\ldots,\,j-1\}$, $p_{x+k}>0$ for all $l\leqslant k \leqslant n+l-1$, then
\begin{align*}
&^m_{l*n_1|}\left(I^{(j)}\bar{A}\right)_{x:\actuarialangle{n}}^1:=
\mathbb{E}\left([j\cdot T_x+1]\cdot \nu^{T_x}\right)^m\mathbbm{1}_{\{l*n_1\leqslant T_x < n+l*n_1\}}=\\
&
\sum_{k=l}^{n+l-1}\frac{{_k}p_x\cdot\nu^{mk}\cdot\log p_{x+k}}{\log(\nu^m\cdot p_{x+k})}
\left(1-\nu^\frac{m}{j}\left(p_{x+k}\right)^\frac{1}{j}\right)
\sum_{d=n_1}^{j-1}\left(d+jk+1\right)^m\nu^{\frac{dm}{j}}
\left(p_{x+k}\right)^\frac{d}{j}\\
&+\frac{\nu^{(n+k)m}\cdot{_{n+l}}p_x\cdot\log p_{x+n+l}\cdot\left(1-\left(p_{x+n+l}\right)^{\frac{1}{j}}\cdot\nu^{\frac{m}{j}}\right)}{\log(\nu^m\cdot p_{x+n+l})}\times\\
&\times\sum_{d=0}^{n_1-1}(d+j\cdot(n+l)+1)^m
\nu^{\frac{dm}{j}}
\left(p_{x+n+l}\right)^\frac{d}{j}.
\end{align*}
\end{proposition}

{\sc Remark 6}: {\it If $p_{x+k}\to0$ in Proposition \ref{prop:j_[T]_nu_x}, then $\log p_{x+k}/\log (\nu^m\cdot p_{x+k})\to1$. If $\nu^m\cdot p_{x+k}\to1$, then
\begin{align*}
\frac{1-\nu^{m/j}(p_{x+k})^{1/j}}{\log(\nu^m\cdot p_{x+k})}\to-\frac{1}{j}.
\end{align*}
} 

\bigskip

As mentioned at the beginning of this section, formulas of Propositions \ref{prop:main}--\ref{prop:j_[T]_nu_x} can be modified to reflect on different types of insurance products. For example, the present value of the yearly decreasing insured amount is characterized as follows 
\begin{align*}
&\mathbb{E}\left(\nu^{T_x}\left(n+l-[T_x]\right)\right)^m\mathbbm{1}_{\{l\leqslant T_x<n+l\}}
=\sum_{k=l}^{l+n-1}\int_{k}^{k+1}\nu^{tm}(n+l-k)^m f_x(t)\,dt\\
&=\sum_{k=l}^{l+n-1}(n+l-k)^m\cdot \frac{{_k}p_x}{(p_{x+k})^k}\cdot\log\frac{1}{p_{x+k}}
\int_{k}^{k+1}\nu^{tm}\cdot(p_{x+k})^t\,dt,
\end{align*}
where the last integral is the same as \eqref{the_same}, used in Proposition \ref{prop:[T]_nu}.

\section{Examples}\label{sec:examples}

To convince the correctness of the formulas given in Section \ref{sec:results}, we select data from the human mortality database \cite{HMD_LTU_2025} and check the outputs of the presented formulas. We also give comparisons of the same characteristics computed under the (UDD) and (B) assumptions. The corresponding formulas under Balducci's mortality assumption (B) can be found in \cite{Balducci}, while the same under (UDD) are mainly derived in \cite{bowers1986actuarial}. All the presented computations are implemented using the software \cite{Mathematica,Rsoftware}, while visualizations are produced using \texttt{ggplot2} \cite{ggplot2}.  

\bigskip

\begin{ex}\label{ex1}
Say that a 50-year-old person buys a life insurance policy for nine years with a two-year deferment. We assume the yearly interest rate of $i=3\%$ and compute the net single premiums and the second moments of the corresponding random variables under assumptions (UDD), (C), and (B).
\end{ex}

\bigskip

In this example, $n=7$, $l=2$, $\nu=1/1.03$, and the mortality data are
\begin{table}[h!]
\centering
\begin{tabular}{|c|c|c|c|c|c|c|c|c|c|c|}
\hline
$x$ & $50$ & $51$ & $52$ & $53$&$54$&$55$&$56$&$57$&$58$&$59$ \\
\hline
$l_x$ & 94058 & 93563 &  93048 & 92500&91866&91228&90450&89649&88868&88107 \\
\hline
\end{tabular}
\caption{The passage from Lithuanian mortality data table of 2024 \cite{HMD_LTU_2025}; both sexes; $l_0=100000$.}
\label{table:1}
\end{table}

According to the selected data, we obtain Table \ref{table:2}

\begin{table}[h!]
\centering
\def\arraystretch{1.5}
\begin{tabular}{|c|c|c|c|c|} 
\hline
Expectation & Proposition for (C)& (UDD) & (C) & (B)\\ 
\hline
${_{2|}}\bar{A}^1_{50:\actuarialangle{7}}$ & \ref{prop:main} & 0.0444324 & 0.0444333& 0.0444342\\ 
\hline
$^2_{2|}\bar{A}_{50:\actuarialangle{7}}^1$ & \ref{prop:main} & 0.0377097 & 0.0377111 &0.0377126\\
 \hline
${{_{2|}}\mathop{e}\limits^\circ}_{50:\actuarialangle{7}}$ & \ref{prop:T_x} & 0.3005752 & 0.3005404 & 0.3005352\\
\hline
$\mathbb{E}T_{50}^2\mathbbm{1}_{\{2\leqslant T_{50}< 9\}}$ & \ref{prop:T_x} & 1.9223564 & 1.9219430& 1.9394924 \\
\hline
$_{2|}\left(\bar{I}\bar{A}\right)_{50:\actuarialangle{7}}^1$ & \ref{increasing} & 0.2491531 & 0.2491289 & 0.2492013\\ 
\hline
$^2_{2|}\left(\bar{I}\bar{A}\right)_{50:\actuarialangle{7}}^1$ & \ref{increasing} & 1.2843333 & 1.2841040 & 1.2956242\\ 
\hline
$_{2|}\left(I\bar{A}\right)_{50:\actuarialangle{7}}^1$ & \ref{prop:[T]_nu} & 0.2714787 & 0.2714842 &0.2736986\\ 
\hline
$^2_{2|}\left(I\bar{A}\right)_{50:\actuarialangle{7}}^1$ & \ref{prop:[T]_nu} & 1.4999712 & 1.5000320 &1.5124412\\ 
\hline
\end{tabular}
\caption{Characteristics of the net single premiums according to the data of Example \ref{ex1} and three different interpolations of $s(x)$.}
\label{table:2}
\end{table}
\newpage
If, in addition, $j=12$ and $n_1=0$ (monthly intervals with deferment of two years), then we obtain Table \ref{table:3}.

\begin{table}[h!]
\centering
\def\arraystretch{1.5}
\begin{tabular}{|c|c|c|c|c|} 
\hline
Expectation & Proposition for (C)& (UDD) & (C) & (B)\\ 
\hline
$_{2*0|}\left(A^{(12)}\right)_{50:\actuarialangle{7}}^1$ & \ref{prop:j_times} & 	
0.04437773 & 0.04437859 &0.04472723\\ 
\hline
$^2_{2*0|}\left(A^{(12)}\right)_{50:\actuarialangle{7}}^1$ & \ref{prop:j_times} & 0.03761687 & 0.03761831 &  0.03790957 \\ 
\hline
$_{2*0|}\left(I^{(12)}\bar{A}\right)_{50:\actuarialangle{7}}^1$ & \ref{prop:j_[T]_nu_x} & 3.01206234 &	
3.01177542& 3.03746723\\ 
\hline
$^2_{2*0|}\left(I^{(12)}\bar{A}\right)_{50:\actuarialangle{7}}^1$ & \ref{prop:j_[T]_nu_x} & 187.437832 &187.404907&	
189.083904\\ 
\hline
\end{tabular}
\caption{Characteristics of the net single premiums according to the data of Example \ref{ex1} and three different interpolations of $s(x)$.}
\label{table:3}
\end{table}

\newpage
\begin{ex}
Say that an $x$-years-old person, where $x \in \{18,19,\dots,70\}$, buys a life insurance policy for nine years with a two-year deferment and the unit payoff is immediate after the insured's death. We assume the yearly interest rate of $i=3\%$ and compute the net single premium ${}^{1}_{2\mid}\,\overline{A}_{x:\actuarialangle{7}}$ under assumptions (UDD), (B), and (C), when the insurer's age $x$ varies from $18$ to $70$.
\end{ex}

\bigskip

In this example, $n=7$, $l=2$, $\nu=1/1.03$, and the provided partial mortality data are
\begin{table}[h!]
\centering
\begin{tabular}{|c|c|c|c|c|c|c|c|c|}
\hline
$x$ & $18$ & $19$ & $\ldots$ & $70$&$71$&$\ldots$&$78$&$79$ \\
\hline
$l_x$&99461&99421&$\ldots$&73798&72062&$\ldots$&57442&55056 \\
\hline
\end{tabular}
\caption{The passage from Lithuanian mortality data table of 2024 \cite{HMD_LTU_2025}; both sexes; $l_0=100000$.}
\label{table:4}
\end{table}

In Figures \ref{fig:Ax7_0}, \ref{fig:Ax7}, and \ref{fig:Ax7_1} below, we use the interpolations of (UDD), (C), and (B) to illustrate the growth of the net single premiums as the insurer's age increases. 

\begin{figure}[H]
    \centering
    \includegraphics[width=0.9\textwidth]{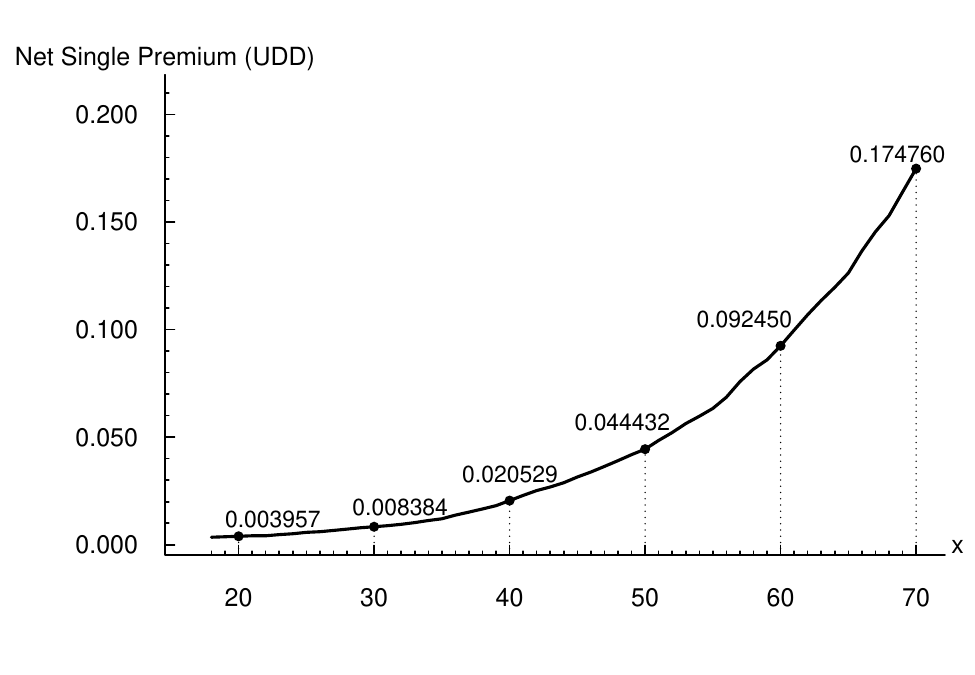}
    \caption{The values of ${}^{1}_{2\mid}\,\overline{A}_{x:\actuarialangle{7}}$ for $x \in \{18,\dots,70\}$ under the assumption~(UDD).}
    \label{fig:Ax7_0}
\end{figure}

\begin{figure}[H]
    \centering
    \includegraphics[width=0.9\textwidth]{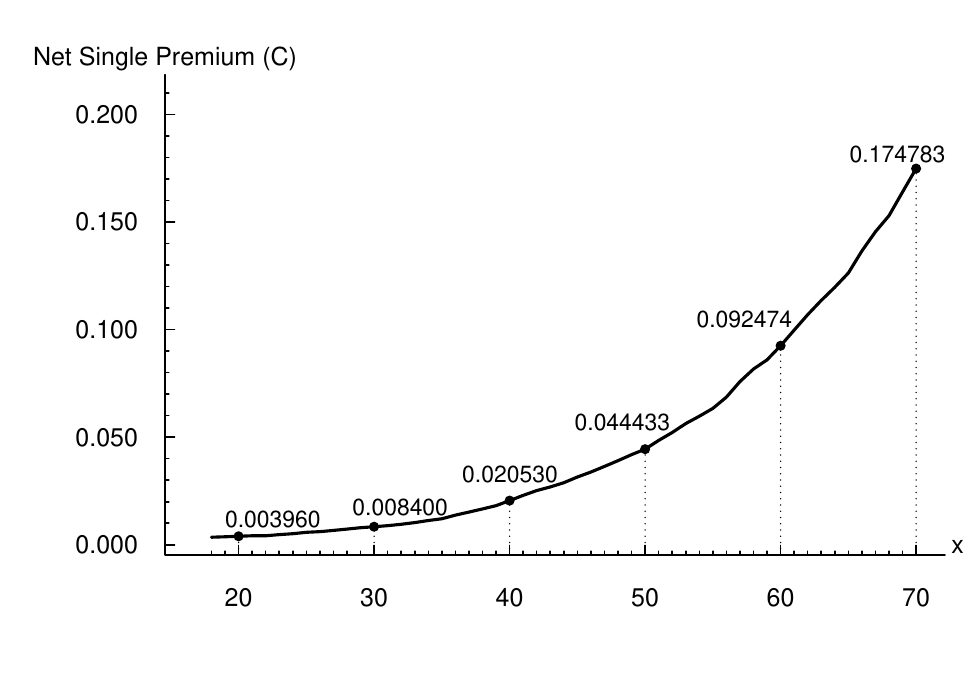}
    \caption{The values of ${}^{1}_{2\mid}\,\overline{A}_{x:\actuarialangle{7}}$ for $x \in \{18,\dots,70\}$ under the assumption~(C).}
    \label{fig:Ax7}
\end{figure}

\begin{figure}[H]
    \centering
    \includegraphics[width=0.9\textwidth]{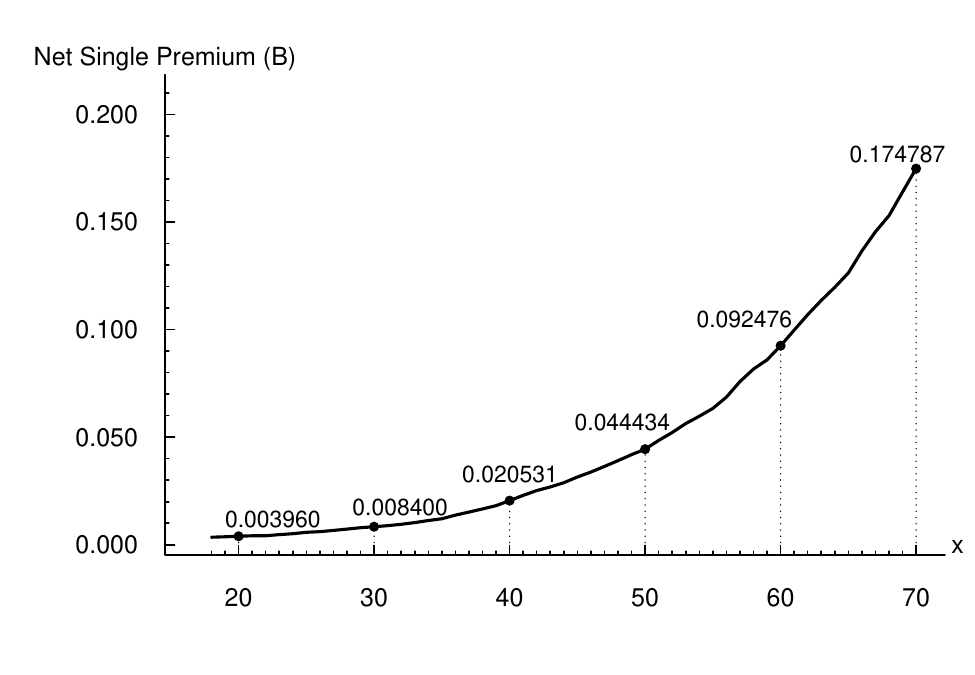}
    \caption{The values of ${}^{1}_{2\mid}\,\overline{A}_{x:\actuarialangle{7}}$ for $x \in \{18,\dots,70\}$ under the assumption~(B).}
    \label{fig:Ax7_1}
\end{figure}

\begin{ex}\label{ex3}
Let $l=1$, $x=0$, $i=3\%$, $n\to\infty$, and consider a Gompertz survival model with parameters $\alpha=0.09$ and $\beta=0.0007$, and the continuos survival function
\begin{align}\label{Gompertz}
s(u)=\exp\!\left(-\frac{\beta}{\alpha}\left(e^{\alpha u}-1\right)\right),\,u\geqslant0.
\end{align}
We take the discrete points $s(0),\,s(1),\,\ldots$ from \eqref{Gompertz}, connect them according to (UDD), (C), (B), and compute the expectations of type as in Propositions \ref{prop:main}--\ref{prop:j_[T]_nu_x} when
$m=1$ and $m=2$. Moreover, as the survival function in \eqref{Gompertz} is continuous, we compute the same expectations using no interpolation.
\end{ex}

\bigskip
The Gompertz and Gompertz–Makeham survival laws are among the top widely used parametric mortality models, see, for example, \cite{CMSTM}.

According to \eqref{Gompertz}, the $k$-year survival probability for age $x$ is
\begin{align}\label{Gompertz_kpx}
{_k}p_x 
= \frac{s(x+k)}{s(x)}
= \exp\!\left(
     -\frac{\beta}{\alpha}e^{\alpha x}\left(e^{\alpha k}-1\right)
  \right),
\qquad x,\,k\in\mathbb{N}_0.
\end{align}
Consequently, the one-year survival probability at age $x+k\in\mathbb{N}_0$ is
\begin{align}\label{Gompertz_pxk}
p_{x+k} 
= \exp\!\left(
    -\frac{\beta}{\alpha}e^{\alpha(x+k)}\left(e^{\alpha}-1\right)
  \right).
\end{align}

The corresponding force of mortality in the Gompertz model \eqref{Gompertz} is 
\begin{equation}\label{Gompertz_force}
\mu(u)
= -\frac{d}{du}\log s(u)
= \beta e^{\alpha u},\,u\geqslant0
\end{equation}
and the probability density function of the newborn's future lifetime $T_0$ is
\begin{equation}\label{Gompertz_density}
f_0(u)
= \mu(u)\,s(u)
= \beta e^{\alpha u}
\exp\!\left(-\frac{\beta}{\alpha}\left(e^{\alpha u}-1\right)\right),
\qquad u \geqslant 0.
\end{equation}

In Figures \ref{fig:survival} and \ref{fig:pmf} we depict $s(x)$ and $s(x+1)-s(x)$, $x\in\mathbb{N}_0$ obtained from \eqref{Gompertz}.

\begin{figure}[H]
    \centering
    \includegraphics[width=0.8\textwidth]{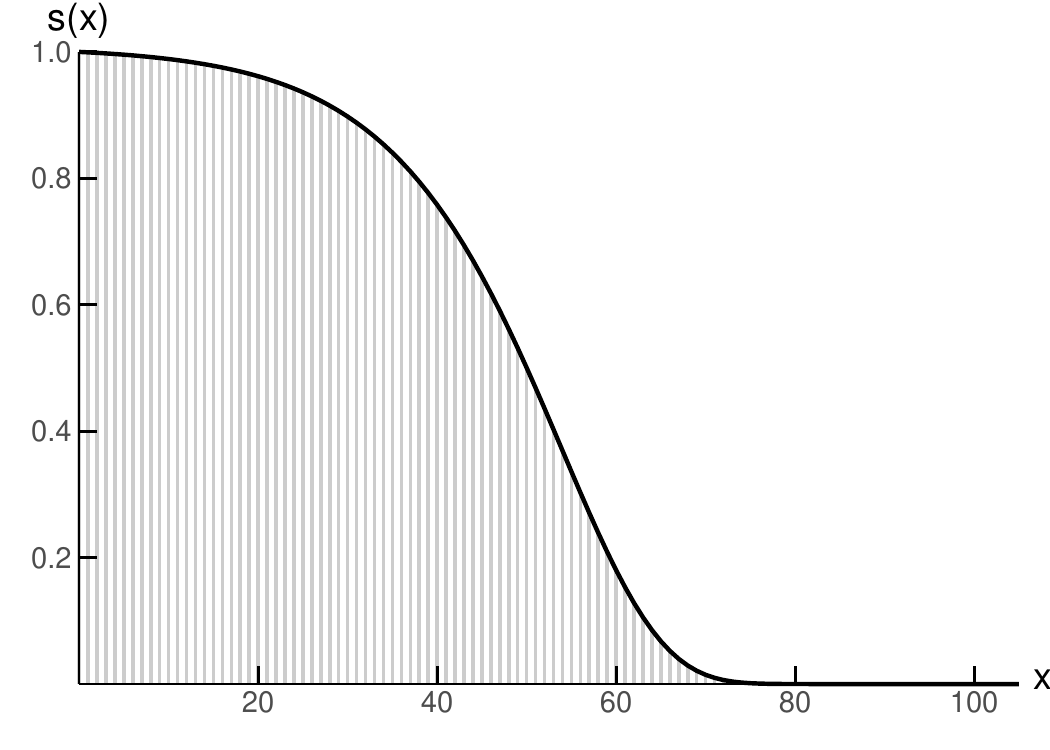}
    \caption{The Gompertz survival function 
    $s(x)=\exp\!\left(-\frac{\beta}{\alpha}\left(e^{\alpha x}-1\right)\right)$,
    $x\in\mathbb{N}_0$, $\alpha=0.09$, $\beta=0.0007$.}
    \label{fig:survival}
\end{figure}

\begin{figure}[H]
    \centering
    \includegraphics[width=0.8\textwidth]{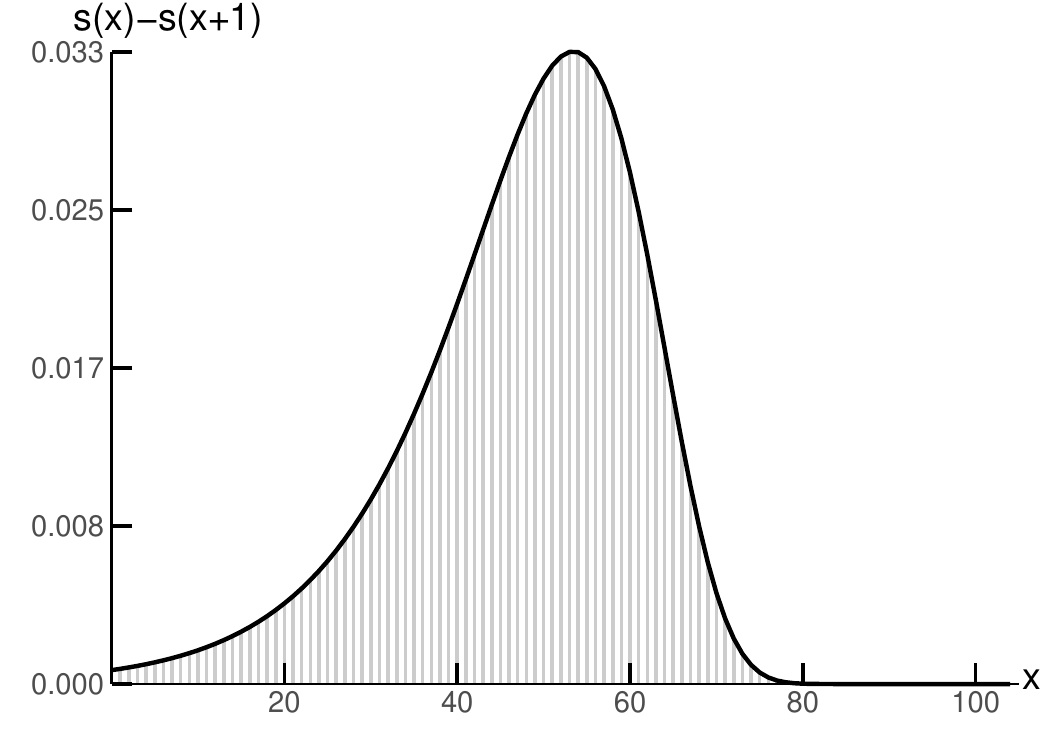}
    \caption{The Gompertz probability mass function
    $\mathbb{P}(X=x)=s(x)-s(x+1)$, $x\in\mathbb{N}_0$, $\alpha=0.09$, $\beta=0.0007$.}
    \label{fig:pmf}
\end{figure}

In Figure \ref{f1} below, we set $x=78,\,79,\,80$ into \eqref{Gompertz} and connect these points according to (UDD), (C), (B), and (G) itself, where (G) means \eqref{Gompertz}. One may observe from Figure \ref{f1} that the constant force of mortality interpolation (C) is very close to the continuous Gompertz survival function.

\begin{figure}[H]
\centering
\includegraphics[scale=0.8]{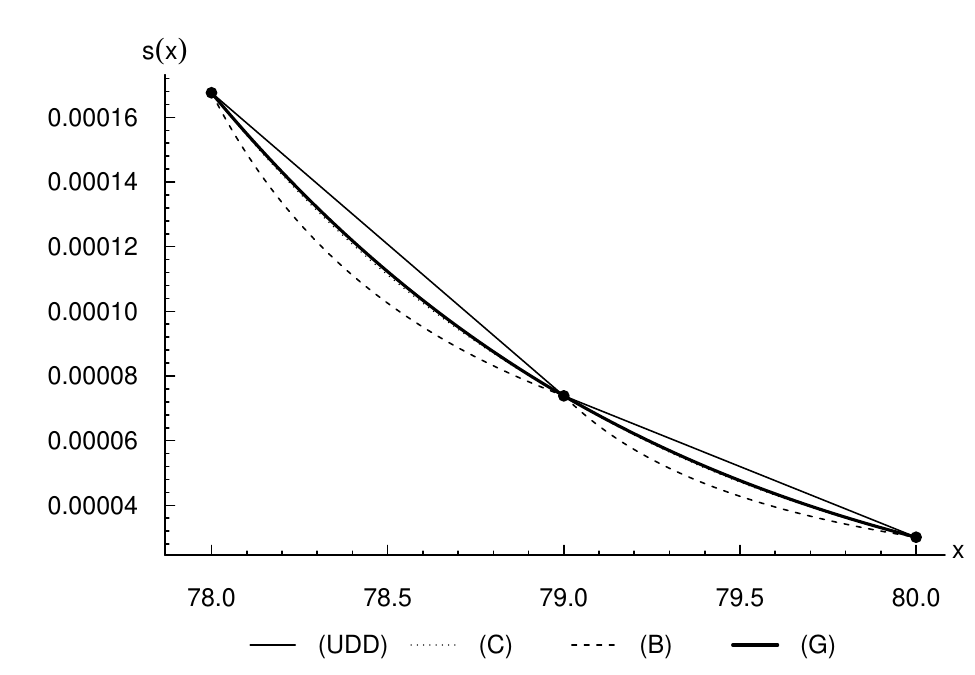}
\caption{
Interpolation of the Gompertz survival function $s(u)$ over ages 78--80 under (UDD), (B) and (C) assumptions, compared with the exact continuous Gompertz survival law (G).
}
\label{f1}
\end{figure}

In Figure \ref{f2} below, we depict derivatives $-s'(u)$ of the functions in Figure \ref{f1}. While all interpolations reproduce the same annual death probability, they allocate the probability mass differently within the year; the (C) assumption closely matches the Gompertz density, whereas (UDD) and (B) display deviations.

\begin{figure}[H]
\centering
\includegraphics[scale=0.8]{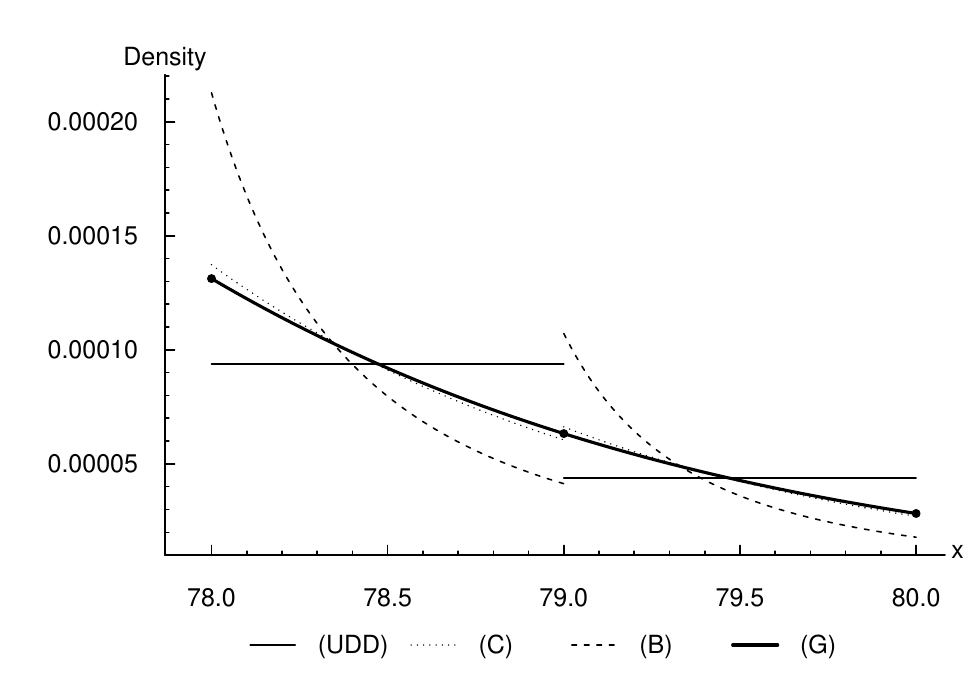}
\caption{
Interpolation of the probability density function $f_0(u)$ over ages 78--80 implied by the (UDD), (B), and (C) assumptions, compared with the exact Gompertz density (G).
}
\label{f2}
\end{figure}

According to the selected data in Example \ref{ex3}, obtain Table \ref{table:5}.

\begin{table}[h!]
\centering
\def\arraystretch{1.5}
\begin{tabular}{|c|c|c|c|c|c|} 
\hline
Expectation & Proposition for (C)& (UDD) & (C) & (B) & (G)\\ 
\hline
${_{1|}}\bar{A}_{0}$ & \ref{prop:main} & 0.2627886 & 0.2628295& 0.2629811 & 0.2627713\\ 
\hline
$^2_{1|}\bar{A}_{0}$ & \ref{prop:main} & 0.0843555 & 0.0843719 & 0.0849373 & 0.0843345\\
 \hline
${{_{1|}}\mathop{e}\limits^\circ}_{0}$ & \ref{prop:T_x} & 48.035677 & 48.028149 & 48.020770 & 48.005241\\
\hline
$\mathbb{E}T_{0}^2\mathbbm{1}_{\{1\leqslant T_{0}\}}$ & \ref{prop:T_x} & 2481.3084 & 2480.4217&2480.1802 & 2481.1415\\
\hline
$_{1|}\left(\bar{I}\bar{A}\right)_{0}^1$ & \ref{increasing} & 11.0544440 & 11.0538410 &11.0526210 & 11.054871\\ 
\hline
$^2_{1|}\left(\bar{I}\bar{A}\right)_{0}^1$ & \ref{increasing} & 123.755820 & 123.772240 &123.719034 & 123.76314\\ 
\hline
$_{1|}\left(I\bar{A}\right)_{0}^1$ & \ref{prop:[T]_nu} & 11.1861217 & 11.1884370 & 11.1899573 & 11.185961\\ 
\hline
$^2_{1|}\left(I\bar{A}\right)_{0}^1$ & \ref{prop:[T]_nu} & 126.719683 & 126.768167 &126.771983 & 126.71242\\ 
\hline
\end{tabular}
\caption{Characteristics of the net single premiums according to the data of Example \ref{ex3} and different interpolations of $s(u)$.}
\label{table:5}
\end{table}

If, in addition, $j=12$ and $n_1=0$, then we obtain Table \ref{table:6} below.

\begin{table}[h!]
\centering
\def\arraystretch{1.5}
\begin{tabular}{|c|c|c|c|c|c|} 
\hline
Expectation & Proposition for (C)& (UDD) & (C) & (B) & (G)\\ 
\hline
$_{1*n_1|}\left(A^{(j)}\right)_{0}^1$ & \ref{prop:j_times} & 0.2624651 & 0.2625057 &0.2628445 & 0.2624479\\ 
\hline
$^2_{1*n_1|}\left(A^{(j)}\right)_{0}^1$ & \ref{prop:j_times} & 0.0841479 & 0.0841641 &0.0876469 & 0.0841271\\ 
\hline
$_{1*n_1|}\left(I^{(j)}\bar{A}\right)_{0}^1$ & \ref{prop:j_[T]_nu_x} & 19.124872 &19.139761&19.143942&19.124635\\ 
\hline
$^2_{1*n_1|}\left(I^{(j)}\bar{A}\right)_{0}^1$ & \ref{prop:j_[T]_nu_x} & 237.54759 &237.55632&237.81754&237.51273\\ 
\hline
\end{tabular}
\caption{Characteristics of the net single premiums according to the data of Example \ref{ex3} and different interpolations of $s(x)$.}
\label{table:6}
\end{table}

\section{Proofs}\label{sec:proofs}

In this Section, we prove all of the statements formulated in Section \ref{sec:results}.

\begin{proof}[Proof of Proposition \ref{prop:main}]
The proof is straightforward
\begin{align*}
\mathbb{E}\nu^{mT_x}\mathbbm{1}_{\{l\leqslant T_x < l+n\}}&=\int_{l}^{l+n}\nu^{mt}f_x(t)\,dt
=\sum_{k=l}^{n+l-1}{_k}p_x\cdot\log \frac{1}{p_{x+k}}\int_{k}^{k+1}\nu^{mt}\cdot (p_{x+k})^{t-k}\,dt\\
&=\sum_{k=l}^{n+l-1}\frac{\nu^{mk}\cdot{_k}p_x\cdot(1-\nu^m\cdot p_{x+k})\cdot \log p_{x+k}}{\log (\nu^m \cdot p_{x+k})}.
\end{align*}
\end{proof}

\begin{proof}[Proof of Proposition \ref{prop:T_x}]
As previously,
\begin{align*}
\mathbb{E}(T_x)^m\mathbbm{1}_{\{l\leqslant T_x< n+l\}}
&=\sum_{k=l}^{n+l-1}\frac{{_k}p_x}{(p_{x+k})^k}\log\frac{1}{p_{x+k}}\int_{k}^{k+1}t^m(p_{x+k})^t\,dt\\
&=\sum_{k=l}^{n+l-1}\frac{{_k}p_x\cdot\Gamma_{m,\,k}}{(p_{x+k})^k(\log 1/p_{x+k})^m},
\end{align*}
where
\begin{align*}
\int_{k}^{k+1} t^m (p_{x+k})^t\,dt=\frac{\Gamma\left(1+m,\,k\log1/p_{x+k}\right)-\Gamma\left(1+m,\,(k+1)\log1/p_{x+k}\right)}{\left(\log 1/p_{x+k}\right)^{1+m}},\,k\in\mathbb{N}_0,
\end{align*}
and $\Gamma(a,\,x)$ is the upper incomplete gamma function \eqref{gamma}.
If $m=1$, then
\begin{align}\label{m=1}
\int_{k}^{k+1}t\cdot(p_{x+k})^t\,dt=\frac{(p_{x+k})^t(-1+t\log p_{x+k})}{(\log p_{x+k})^2}\bigg|_k^{k+1}.
\end{align}
If $m=2$, then
\begin{align}\label{m=2}
\int_{k}^{k+1}t^2\cdot(p_{x+k})^t\,dt=\frac{(p_{x+k})^t\left(2-2t\log p_{x+k}+t^2(\log p_{x+k})^2\right)}{(\log p_{x+k})^3}\bigg|_k^{k+1}.
\end{align}
\end{proof}

\begin{proof}[Proof of Proposition \ref{increasing}]
Arguing the same as before, we get
\begin{align*}
\mathbb{E}\,(T_x\nu^{T_x})^m\mathbbm{1}_{\{l\leqslant T_x< n+l\}}=
\sum_{k=l}^{n+l-1}\frac{{_k}p_x}{(p_{x+k})^k}\log\frac{1}{p_{x+k}}
\int_{k}^{k+1}t^m\cdot \nu^{mt}\cdot (p_{x+k})^t\,dt,
\end{align*}
where
\begin{align*}
&\left(\log\frac{1}{\nu^m\cdot p_{x+k}}\right)^{m+1}\int_{k}^{k+1}t^m\cdot (\nu^m p_{x+k})^t\,dt\\
&=\Gamma\left(m+1,\,k\log\frac{1}{\nu^m\cdot p_{x+k}}\right)-\Gamma\left(m+1,\,(k+1)\log\frac{1}{\nu^m \cdot p_{x+k}}\right)=\tilde{\Gamma}_{m,\,k}.
\end{align*}
\end{proof}

\begin{proof}[Proof of Proposition \ref{prop:[T]_nu}]
We have that
\begin{align*}
&\mathbb{E}\,([T_x+1]\nu^{T_x})^m\mathbbm{1}_{\{l\leqslant T_x< n+l\}}
=\sum_{k=l}^{l+n-1}\frac{(k+1)^m{_k}p_x}{(p_{x+k})^k}\log\frac{1}{p_{x+k}}\int_{k}^{k+1}\nu^{mt}(p_{x+k})^t\,dt\\
&=\sum_{k=l}^{l+n-1}\nu^{mk}\cdot(k+1)^m\cdot{_k}p_x\cdot\log\frac{1}{p_{x+k}}\cdot\frac{p_{x+k}\cdot\nu^m-1}{\log(\nu^m\cdot p_{x+k})}
\end{align*}
due to
\begin{align}\label{the_same}
\int_{k}^{k+1}\nu^{mt}(p_{x+k})^t\,dt=\frac{\nu^{mk}(p_{x+k})^k(-1+p_{x+k}\cdot\nu^m)}{\log(\nu^m p_{x+k})}.
\end{align}
\end{proof}

\begin{proof}[Proof of Lemma \ref{lem:no_greater}]
We first prove the second inequalities of \eqref{exp_ineq_1} and \eqref{exp_ineq_2}, i.e., we compare (C) to (B). Let $k\in\mathbb{N}_0$. For any $l_x>0$, $x\in\mathbb{N}_0$ and $k$ such that $q_{x+k}<1$, we have
\begin{align*}
{\mathfrak I}&:=\int_{k}^{k+1}\left({_k}p_x\cdot\log\frac{1}{p_{x+k}}\cdot(p_{x+k})^{t-k}-\frac{_{k+1}p_x\cdot q_{x+k}}{(1-(k+1-t)q_{x+k})^2}\right)g(t)\,dt\\
&=\int_{k}^{k+1}g(t)\,d\left(-{_k}p_x\cdot(p_{x+k})^{t-k}+\frac{_{k+1}p_x}{1-(k+1-t)q_{x+k}}\right)
=-\int_{k}^{k+1}g(t)\,d \,z(t),
\end{align*}
where
\begin{align*}
z(t)={_k}p_x\cdot(p_{x+k})^{t-k}-\frac{_{k+1}p_x}{1-(k+1-t)q_{x+k}},\,k\leqslant t\leqslant k+1.
\end{align*}
It is easy to check $z(k)=z(k+1)=0$. Then, upon the integration by parts,
\begin{align}\label{ineq:estimate}
{\mathfrak I}=-g(k+1)z(k+1)+g(k)z(k)+\int_{k}^{k+1}z(t)\,g'(t)\,dt
=\int_{k}^{k+1}z(t)\,g'(t)\,dt.
\end{align}
Thus, the sign of the integral $\mathfrak{I}$ in \eqref{ineq:estimate} is completely determined by the signs of the functions $z(t)$ and $g'(t)$. Let us show that $z(t)\geqslant0$ for all $k\leqslant t \leqslant k+1$ if $p_{x+k}>0$. The function $z(t)$ can be rewritten as
\begin{align*}
z(t)={_k}p_x\left((p_{x+k})^{t-k}-\frac{p_{x+k}}{t-k+(1-(t-k)p_{x+k})}\right),\,k\leqslant t \leqslant k+1.
\end{align*}
Under the change of variables $p_{x+k}=a$ and $t-k\mapsto t$, we see that the statement $z(t)\geqslant0$ for all $k\leqslant t \leqslant k+1$ is equivalent to
\begin{align}\label{equiv_ineq}
a^t(1-t)+a^{t-1}t\geqslant1,
\end{align}
for all $0\leqslant t \leqslant 1$, when $0<a\leqslant 1$ is fixed. Let us prove inequality \eqref{equiv_ineq}. By multiplying both sides of \eqref{equiv_ineq} by $a^{1-t}$ and changing the variable $1-t=s$ we get
\begin{align*}
f(s):=a^s\leqslant as+1-s,\,0\leqslant s\leqslant 1,
\end{align*}
where the obtained inequality is Jensen's inequality for the convex function \cite{jensen}, i.e., $f''(s)=(\log a)^2a^s\geqslant0$ for all $0\leqslant s \leqslant1$ and $f(s\cdot1+(1-s)\cdot0)\leqslant sf(1)+(1-s)f(0)$. In conclusion, the sign of the integral $\mathfrak{I}$ in \eqref{ineq:estimate} is determined only by $g'(t)$.

We now prove the first inequalities of \eqref{exp_ineq_1} and \eqref{exp_ineq_2}, i.e., we compare (C) to (UDD). Arguing the same as before, we consider the integral 
\begin{align*}
\tilde{{\mathfrak I}}&:=\int_{k}^{k+1}\left({_k}p_x\cdot\log\frac{1}{p_{x+k}}\cdot(p_{x+k})^{t-k}-\frac{d_{x+k}}{l_x}\right)g(t)\,dt\\
&=\int_{k}^{k+1}g(t)\,d\left(-{_k}p_x\cdot(p_{x+k})^{t-k}-\frac{d_{x+k}}{l_x}t\right)
=-\int_{k}^{k+1}g(t)\,d \,\tilde{z}(t),
\end{align*}
where
\begin{align*}
\tilde{z}(t)={_k}p_x\cdot(p_{x+k})^{t-k}+\frac{d_{x+k}}{l_x}t\geqslant 0,\,k\leqslant t \leqslant k+1.
\end{align*}
Upon the integration by parts,
\begin{align*}
\tilde{{\mathfrak I}}&=g(k)\tilde{z}(k)-g(k+1)\tilde{z}(k+1)+\int_{k}^{k+1}\tilde{z}(t)g'(t)dt\\
&=\left(g(k)-g(k+1)\right)\left({_k}p_x+\frac{d_{x+k}}{l_x}k\right)+\int_{k}^{k+1}\tilde{z}(t)g'(t)dt.
\end{align*}
Notice that $\tilde{z}(t)$ is non-concave. Indeed,
\begin{align*}
\tilde{z}''(t)={_k}p_x\cdot\left(\log p_{x+k}\right)^2\cdot(p_{x+k})^{t-k}\geqslant0,
\end{align*}
for all $k\leqslant t\leqslant k+1$. This, together with the fact that $\tilde{z}(t)$ is non-negative, implies
\begin{align*}
\max_{k\leqslant t \leqslant k+1}\tilde{z}(t)=\tilde{z}(k)=\tilde{z}(k+1)={_k}p_x+\frac{d_{x+k}}{l_x}k.
\end{align*}
Thus, if $g'(t)\leqslant0$, then
\begin{align*}
\tilde{{\mathfrak I}}\geqslant\left(g(k)-g(k+1)\right)\left({_k}p_x+\frac{d_{x+k}}{l_x}k\right)
+\max_{k\leqslant t \leqslant k+1}\tilde{z}(t)\int_{k}^{k+1}g'(t)dt=0.
\end{align*}
And conversely, if $g'(t)\geqslant0$, then
\begin{align*}
\tilde{{\mathfrak I}}\leqslant\left(g(k)-g(k+1)\right)\left({_k}p_x+\frac{d_{x+k}}{l_x}k\right)
+\max_{k\leqslant t \leqslant k+1}\tilde{z}(t)\int_{k}^{k+1}g'(t)dt=0.
\end{align*}
\end{proof}

\begin{proof}[Proof of Lemma \ref{lem:j_times}]
Since $x+k\in\mathbb{N}_0$, $d/j\in[0,\,1)$, and $(d+1)/j\in(0,\,1]$, we then apply the interpolation \eqref{C} and obtain
\begin{align*}
s\left(x+k+\frac{d}{j}\right)-s\left(x+k+\frac{d+1}{j}\right)
=s(x+k)\left(p_{x+k}\right)^{\frac{d}{j}}-s(x+k)\left(p_{x+k}\right)^{\frac{d+1}{j}}.
\end{align*}
The claimed formula follows by dividing both sides of the last equality by $s(x)>0$.
\end{proof}

\begin{proof}[Proof of Proposition \ref{prop:j_times}]
We have that
\begin{align*}
&\mathbb{E}\left(\nu^{\frac{[T_x\,j]+1}{j}}\right)^m\mathbbm{1}_{\{l*n_1\leqslant T_x<n+l*n_1\}}\\
&=\sum_{k=l}^{n+l-1}\sum_{d=n_1}^{j-1}\nu^{(k+(d+1)/j)m}\left(_{k+\frac{d}{j}}p_x-_{k+\frac{d+1}{j}}p_x\right)\\
&+
\nu^{(n+l)m}\sum_{d=0}^{n_1-1}\nu^{(d+1)/j\cdot m}\left(_{n+l+\frac{d}{j}}p_x-_{n+l+\frac{d+1}{j}}p_x\right),
\end{align*}
and the rest follows by Lemma \ref{lem:j_times}.
\end{proof}

\begin{proof}[Proof of Proposition \ref{prop:j_[T]_nu_x}]
We have that
\begin{align*}
&\mathbb{E}\left([j\cdot T_x+1]\cdot \nu^{T_x}\right)^m\mathbbm{1}_{\{l*n_1\leqslant T_x < n+l*n_1\}}\\
&=\sum_{k=l}^{n+l-1}\sum_{d=n_1}^{j-1}(d+j\cdot k+1)^m\int\limits_{k+d/j}^{k+(d+1)/j}\nu^{m t} f_x(t)\,dt\\
&+\sum_{d=0}^{n_1-1}(d+j\cdot(n+l)+1)^m\int\limits_{n+l+d/j}^{n+l+(d+1)/j}\nu^{mt}f_x(t)\,dt,
\end{align*}
where the two involved integrals are
\begin{align*}
&\int\limits_{k+d/j}^{k+(d+1)/j}\nu^{m t} f_x(t)\,dt
=\frac{{_k}p_x}{\left(p_{x+k}\right)^k}\cdot\log\frac{1}{p_{x+k}}\int\limits_{k+d/j}^{k+(d+1)/j}\nu^{m t}\cdot(p_{x+k})^t\cdot\,dt\\
&={_k}p_x\cdot\log\frac{1}{p_{x+k}}\cdot\frac{\left(p_{x+k}\right)^{\frac{d}{j}}\cdot\nu^{\left(k+\frac{d}{j}\right)m}\cdot\left(-1+\left(p_{x+k}\right)^{\frac{1}{j}}\cdot\nu^{\frac{m}{j}}\right)}{\log(\nu^m\cdot p_{x+k})},
\end{align*}
\begin{align*}
&\int\limits_{n+l+d/j}^{n+l+(d+1)/j}\nu^{m t} f_x(t)\,dt\\
&={_{n+l}}p_x\cdot\log\frac{1}{p_{x+n+l}}\cdot\frac{\left(p_{x+n+l}\right)^{\frac{d}{j}}\cdot\nu^{\left(n+l+\frac{d}{j}\right)m}\cdot\left(-1+\left(p_{x+n+l}\right)^{\frac{1}{j}}\cdot\nu^{\frac{m}{j}}\right)}{\log(\nu^m\cdot p_{x+n+l})}.
\end{align*}
\end{proof}

\section{Concluding remarks}
This work aims to supplement the existing theoretical materials in actuarial mathematics. A good understanding of past data provides a better understanding of the future; however, the past rarely predicts the future exactly. In life insurance, the expected present value of the insured amount is just one component in the final price of insurance. It is difficult to judge which interpolation out of (UDD), (C), and (B) fits best, or even whether some other fractional age assumptions should be used. In the works \cite{JONES2000261}, \cite{jones2002critique}, authors criticize these three classical assumptions ((UDD), (C), (B)) as the ones implying discontinuous force of mortality and density at integer ages. On the other hand, if the limit
\begin{align*}
\lim_{\Delta x\to0}\frac{s(x+\Delta x)-s(s)}{\Delta x}=s'(x)
\end{align*}
exists, then it follows
\begin{align}\label{approx_eq}
s(x+\Delta x)-s(x)\approx \Delta x \cdot s'(x)\text{ and } 1-\frac{l_{x+\Delta x}}{l_x}\approx \Delta x \cdot \mu(x)
\end{align}
when $\Delta x$ is some small fixed number. The approximate equalities \eqref{approx_eq} indicate that the mentioned discontinuity gaps shall be small, especially when $l_{x}\approx l_{x+\Delta x}$. In our opinion, the selection between simplistic interpolations and more complicated ones remains open, and it is rather a matter of practitioners. The examples in Section \ref{sec:examples} show that the survival probabilities ${_{k}}p_x$, $p_{x+k}$ at integre ages are much more important than interpolation; compare the numbers in Tables \ref{table:2}, \ref{table:3}, \ref{table:5}, \ref{table:6}, and Figures \ref{fig:Ax7_0}, \ref{fig:Ax7}, \ref{fig:Ax7_1}.

\bibliography{sn-bibliography}

\end{document}